\documentclass[12pt,a4paper]{article}
\usepackage[intlimits]{amsmath}
\usepackage{amsfonts,amssymb,amscd,amsthm}
\usepackage{cite}

\usepackage[all]{xy}

\def\lra{\longrightarrow}
\def\ind{\operatorname{ind}}

\def\Td{\operatorname{Td}}

\def\tr{\operatorname{tr}}

\def\End{\operatorname{End}}

\def\coker{\operatorname{coker}}

\def\H{{\mathcal H}}

\newtheorem{theorem}{Theorem}[section]
\newtheorem{corollary}[theorem]{Corollary}
\newtheorem{lemma}[theorem]{Lemma}
\newtheorem{proposition}[theorem]{Proposition}

\theoremstyle{definition}

\newtheorem{definition}[theorem]{Definition}

\numberwithin{equation}{section}

\title{Uniformization and an Index Theorem for  Elliptic Operators 
Associated with Diffeomorphisms of a Manifold}
\author{Anton Savin, Elmar Schrohe,  Boris Sternin}

\begin{document}

\maketitle

\tableofcontents

\section*{Introduction}
\addcontentsline{toc}{section}{Introduction}

Let  $M$ be a closed smooth 
Riemannian manifold and   $g:M\to M$ a smooth isometric diffeomorphism. 
The powers of $g$ generate an action of the group   $\mathbb{Z}$
of integers  on $M$.
In this article we consider a class of {\em differential operators   with shifts},
i.e.\ operators  of the form
\begin{equation}\label{eq-intr1}
D=\sum D_k T^k:C^\infty(M)\lra C^\infty(M),  
\end{equation}
where the sum is finite, the $D_k$  are differential operators, and 
$T^ku(x)=u(g^k(x))$.

We shall introduce an ellipticity condition which implies the Fredholm 
property for operators of this type and compute their index.
The method, which we use to solve the problem, is called pseudodifferential (actually even differential) uniformization.  
The idea is to study instead of $D$ an elliptic differential operator
with the same index on a suitably chosen manifold.
This procedure consists of two steps.

In the first step we replace the original manifold  $M$ by $M\times\mathbb{R}$ 
endowed with the diagonal action of   $\mathbb{Z}$, which 
is free and proper. 
On $M\times\mathbb{R}$ we define an elliptic operator 
as the external product of $D$ with a special operator $A$ of index one
on the real line.

In the second step we interpret this operator on $M\times\mathbb{R}$  
as a differential operator in sections of a Hilbert space bundle with fiber 
$l^2(\mathbb Z)$  over the quotient  
$(M\times\mathbb{R})/\mathbb{Z}$, which is a
smooth manifold, since the action is free and proper
(cf.\ \cite{BaCo2,HiKo1, Luk1} for ideas similar in spirit).  
The point of this construction is that the index is preserved while, 
for the resulting operator, 
which we call the {\em differential uniformization} of $D$, 
the action of $g$ simply is a shift in the fiber, which is much easier to treat.
We compute the symbol of the resulting operator, 
establish the Fredholm property 
(finiteness theorem), and express the index in terms of the symbol of $D$ and 
topological invariants of the manifold. 
 
Note that the index problem in a related situation (even for the case, 
when instead of $\mathbb{Z}$ we take an arbitrary discrete group of polynomial growth) was studied in   \cite{NaSaSt17}.  
The solution used there did not use the idea of uniformization and 
was rather complicated,
while the method presented here is quite simple and natural.

\section{Statement of the Main Results}

\paragraph{1. Symbol and ellipticity.} 

Let $M$ be a smooth Riemannian manifold and $g:M\to M$ a metric preserving 
diffeomorphism.   
We consider an operator of the form 
\begin{equation}\label{eq-op1}
D=\sum D_k T^k:C^\infty(M)\lra C^\infty(M)
\end{equation}
with differential operators $D_k$ and the shift operator $T$ defined by 
$Tu(x)=u(g(x))$. 
While  {\em a priori}  the $D_k$ can have any order, our 
construction will simplify, if the order is equal to one. 
Hence, from now on we will {\em assume that the $D_k$ are first order 
differential operators.}
The modifications for higher order operators are left to the reader. 

The definition, below, is in the spirit of Antonevich and Lebedev, 
cf.\  \cite{AnLe1}, Sections 40-42 and Remark 37.12.

\begin{definition}\label{def-1} 
The {\em symbol} $\sigma(D)$ of $D$ 
is the function on $T^*M\setminus 0$ taking values in operators on 
$l^2(\mathbb Z)$, defined by   
\begin{eqnarray}\label{elliptic}
(\sigma(D)(x,\xi))w(n)=
\sum_k \sigma(D_k)(\partial g^{n}(x,\xi))\mathcal{T}^kw(n), 
\quad w\in l^2(\mathbb Z), n\in \mathbb Z,
\end{eqnarray}
where $\mathcal T$ is the shift operator on $l^2(\mathbb Z)$
given by $(\mathcal{T}w)(n)=w(n+1)$, 
$\sigma(D_k)$ is the principal symbol of $D_k$, and
$\partial g=({}^t dg)^{-1}:T^*M\to T^*M$ is the codifferential of $g$.

We call $D$ {\em elliptic}, if the operator in \eqref{elliptic} is invertible for
$(x,\xi)\in T^*M\setminus0$. 
\end{definition}
The aim of the paper is to reduce this operator to a classical elliptic differential 
operator on a closed manifold via a uniformization procedure and to 
compute its index.

\paragraph{2. Reduction to proper action: passage to the infinite cylinder.}
 
On the real line with coordinate $t$, consider the operator 
$$
 A=\frac\partial{\partial  t}+t:
 \H^s(\mathbb{R})\to \H^{s-1}(\mathbb{R})
$$
(see the Appendix for the definition of $\H^s(\mathbb R)$ and 
more details on operators of this type). 
It is elliptic and has index one: The 
cokernel is trivial, while the kernel is generated by the function
$e^{-t^2/2}$.  

We next extend the isometry $g:M\to M$ to an isometry 
$\widetilde g:M\times \mathbb{R}\to M\times \mathbb{R}$ 
of the infinite cylinder $M\times \mathbb{R}$ with local coordinates $x,t$ by
letting 
$$
\widetilde g(x,t)=(g(x), t+1).
$$
As before, we obtain an action of $\mathbb Z$ and a shift operator $\widetilde T$
by
$$
(\widetilde{T}u)(x,t)=u(g(x),t+1).
$$
With the operators $D_k$ from Equation \eqref{eq-op1} 
we then define the operator
\begin{eqnarray*}\label{tildeD}
 \widetilde{D}=\sum_k D_k \widetilde{T}^k:\H^s(M\times \mathbb{R})
 \lra \H^{s-1}(M\times \mathbb{R})
\end{eqnarray*}
and consider the   {\em external product} of the operators 
$\widetilde{D}$ and $A$, 
\begin{equation}\label{eq-product1}
\widetilde{D}\# A=\left(
                   \begin{array}{cc}
                        \widetilde{D} &  A\\
                       -  A^* &   \widetilde{D}^*
                   \end{array} 
                \right):\H^s(M\times\mathbb{R},\mathbb{C}^2)\lra \H^{s-1}(M\times\mathbb{R},\mathbb{C}^2).
\end{equation}
Here we take the adjoint with respect to the  inner product in $L^2(M\times \mathbb{R})$.

\paragraph{3. Reduction to  the (smooth) orbit space.}
The isometry  $\widetilde{g}$ defines a free proper action of the group  
$\mathbb{Z}$ on the cylinder. 
Hence, the corresponding orbit space is a smooth manifold. 
Let us consider functions on the cylinder  as functions on the orbit space 
ranging in functions on the fiber of the projection\footnote{The space $M_\mathbb{Z}$ is also called the {\em homotopy quotient} of $M$ by the action of  $\mathbb{Z}$.}
$$
M\times \mathbb{R}\lra  (M\times\mathbb{R})/\mathbb{Z}=M_\mathbb{Z}.
$$
The point of this representation is  that 
the shift operator $\widetilde{T}$  on $M\times \mathbb{R}$
becomes an operator of multiplication, i.e.,  a local operator, on the space   
$M_\mathbb{Z}$.

For  $L^2$-spaces this representation of functions on $M\times\mathbb{R}$ in terms of functions on $M_\mathbb{Z}$ is described by the following 
proposition.
\begin{proposition}
Let $\mathcal{E}$ be the vector bundle 
$$
\mathcal{E}=(M\times \mathbb{R}\times l^2(\mathbb{Z}))/ \mathbb{Z}
$$
over $M_\mathbb{Z}$ with fiber $l^2(\mathbb{Z})$.
Here $n\in\mathbb{Z}$ acts by
$$
(x,t,u)\longmapsto (g^n(x),t+n,\mathcal{T}^{-n}u).
$$
We then have the isomorphism 
\begin{equation}\label{eq-isoma1}
 \begin{array}{ccc}
   I:L^2(M\times \mathbb{R}) & \lra & L^2(M_\mathbb{Z},\mathcal{E}),\vspace{1mm}\\
(I \varphi)(x,t,n)  &=& \varphi(g^n(x),t+n).
 \end{array}
\end{equation}
\end{proposition}
\begin{proof}
A direct computation shows that the inverse mapping is equal to 
$$
(I^{-1}u)(x,t)=u(g^{-[t]}(x),\{t\},[t]),
$$
where $t=[t]+\{t\}$ is the decomposition into integer and fractional part of  $t$. 
\end{proof}

\paragraph{4. Main theorem.}

\begin{theorem}\label{main}
\begin{enumerate}\renewcommand{\labelenumi}{{\rm (\alph{enumi})}}
\item  The operator $I$ in \eqref{eq-isoma1} induces an isomorphism of 
Sobolev spaces 
\begin{equation}\label{eq-isom3}
 I:\H^s(M\times \mathbb{R})\to H^s(M_\mathbb{Z},\mathcal{E}),
\end{equation}
where the Sobolev space $H^s(M_{\mathbb Z},{\mathcal E})$ 
of sections of the infinite-dimensional bundle  
$\mathcal{E}$ over $M_\mathbb{Z}$
carries the norm 
\begin{eqnarray}\label{norms1}
\|u\|^2=\sum_n\|(1+\Delta_{x,t} + n^2)^{s/2}u(x,t,n)\|^2_{L^2(M\times[0,1])}
\end{eqnarray}
 \item The operator $\mathcal{D}=I (\widetilde{D}\# A)I^{-1}$, which closes the commutative diagram
\begin{equation}\label{eq-diag1}
 \xymatrix{
   \H^s(M\times \mathbb{R}) \ar[r]^{\widetilde{D}\# A} \ar[d]_I &  \H^{s-1}(M\times \mathbb{R})   \ar[d]^I \\
  H^s(M_\mathbb{Z},\mathcal{E}) \ar[r]_{\mathcal{D}} & H^{s-1}(M_\mathbb{Z},\mathcal{E}),
 }
\end{equation}
is a differential operator.
Its symbol in the sense of {\rm \cite{NSScS99}}, Section 3.3,
is the (non-homogeneous) operator-valued function on 
$T^*M_{\mathbb Z}$
\begin{equation}\label{orb-symbol1}
\sigma(\mathcal{D})(x,\xi,t,\tau)=(\sigma(D)(x,\xi))\# (i\tau+t+n):l^2(\mathbb{Z},\mu_{\xi,\tau,s})\lra l^2(\mathbb{Z},\mu_{\xi,\tau,s-1}).
\end{equation}
Here,   $l^2(\mathbb Z, \mu_{\xi,\tau,s})$ carries the norm  
\begin{equation}\label{norms_l2} 
 \|w\|^2_{\xi,\tau,s}=\sum_n(1+\xi^2+\tau^2+n^2)^s |w(n)|^2.
\end{equation}
\item If $D$ is elliptic,  
then the differential operator $\mathcal{D}$ is also elliptic, i.e.\
the operator in  \eqref{orb-symbol1} is invertible for large $|(\xi,\tau)|$, say for 
$|(\xi,\tau)|\ge R$.

\item Under the assumptions of  (c) one has the equality: 
\begin{equation}\label{eq-eq1}
 \ind D=\ind \mathcal{D}.
\end{equation}
\end{enumerate}
\end{theorem}

This theorem will be proven in Sections~\ref{secAA} and \ref{secBB}. 
We will first establish assertions (a)-(c) and then show equality (d) for the indices.

\paragraph{5. An index formula.}

We denote by  $S^*M_{\mathbb{Z}}$ the cosphere bundle of $M_{\mathbb{Z}}$  of radius $R$ with $R$ as in Theorem \ref{main}(c) 
and consider the space  $\Lambda(S^*M_{\mathbb{Z}},\End\mathcal{E})$ 
of differential forms on   $S^*M_{\mathbb{Z}}$ 
ranging in endomorphisms of the Hilbert bundle  $\mathcal{E}$. 
This space is endowed with the differential  $d$ 
(the differential is well defined, since   $\mathcal{E}$ is flat). 
Taking the fiberwise trace of operators in  
$l^2(\mathbb{Z})$ gives the (partially defined) mapping
$$
\tr_\mathcal{E}: \Lambda(S^*M_{\mathbb{Z}},\End\mathcal{E})\lra \Lambda(S^*M_{\mathbb{Z}}).
$$
\begin{definition}\label{def-4z}{\em The topological index} of an elliptic operator $\mathcal{D}$ is the number
\begin{equation}\label{eq-indt2c} 
 \ind_t  \mathcal{D} =\sum_j C_j\int_{S^*  M_{\mathbb{Z} }}\tr_\mathcal{E}
\left[(\sigma(\mathcal{D}) ^{-1}d\sigma(\mathcal{D}))^{2j-1}\Td(T^*_\mathbb{C}M) \right] ,
\end{equation}
where $C_j= {(j-1)!}/[{(2\pi i)^j(2j-1)!}]$.

The properties of the topological index will be studied in Section 4.2, below.\end{definition}

We endow $M_\mathbb{Z}$  with the metric $h+dt^2$, where $h$ is the $g$-invariant metric on  $M$. 
Since $g$ is an isometry,
this metric is well-defined and we have the equality  
$$
\Td(T^*_\mathbb{C}M)=\Td(T^*_\mathbb{C}M_\mathbb{Z})
$$
of the differential forms, 
which represent the Todd classes of the complexification of 
the cotangent bundles of   $M$ and $M_\mathbb{Z}$, respectively.

\begin{theorem}\label{th-luke1a}
For the elliptic operator $\mathcal{D}$ in \eqref{eq-diag1}
with symbol $\sigma(\mathcal{D})$  the topological index 
\eqref{eq-indt2c} is well defined provided that $n>1$,
and one has the index formula 
\begin{equation}\label{ind-luke1a}
 \ind \mathcal{D}=\ind_t  \mathcal{D}.
\end{equation}
\end{theorem}
Equalities  \eqref{eq-eq1} and \eqref{ind-luke1a} give the index formula for the original operator $D$.
Theorem \ref{th-luke1a} will be proven in Section \ref{sec2}.

\section{Reduction to the Orbit Space}\label{secAA}

\paragraph{1. Isomorphism of Sobolev spaces.}

In order to see that the mapping~\eqref{eq-isom3} is an isomorphism, 
we shall use the two lemmas, below. 

Let   $\gamma$ be the one-dimensional complex vector bundle over   
$M_\mathbb{Z}\times \mathbb{S}^1$ defined by the equality
\begin{equation*}\label{bottbundle0}
C^\infty(M_\mathbb{Z} \times \mathbb{S}^1,\gamma)=\{ v\in C^\infty(M\times \mathbb{R}\times \mathbb{S}^1)
\; |\; v(g(x),t+1,\varphi)=v(x,t,\varphi)e^{-  i\varphi}\}.
\end{equation*}

\begin{lemma}\label{lem1}
The mapping
$$
 \begin{array}{ccc}
   K:H^s(M_\mathbb{Z},\mathcal{E}) & \lra  &H^s(M_\mathbb{Z}\times \mathbb{S}^1,\gamma)\vspace{1mm}\\
 (u(x,t,n))_n & \longmapsto & \mathcal{F}_{n\to\varphi} u(x,t,\cdot)=\frac 1{\sqrt{2\pi}}\sum_n u(x,t,n)e^{in\varphi},
 \end{array}
$$
which is just the fiberwise Fourier transform (series), is an isomorphism for all   $s$. 
\end{lemma}
\begin{proof}
This is an immediate consequence of the fact that the Fourier transform 
$\mathcal F_{n\to \varphi}$ takes the norm \eqref{norms1} defining
$H^s(M_\mathbb{Z},\mathcal{E})$, 
to the  norm
\begin{equation}\label{norms2}
 \|u\|_{s}=\left\|\left(1+\Delta_{x,t}-\frac{\partial^2}{\partial\varphi} \right)^{s/2} u(x,t,\varphi)\right\|_{L^2(M_{\mathbb Z}\times \mathbb{S}^1,\gamma)} 
\end{equation}
defining $H^s(M_\mathbb{Z}\times \mathbb{S}^1,\gamma)$, 
if we consider this space
as the space of functions on $M_\mathbb{Z}$ 
taking values in $\mathbb{S}^1$, 
cf.\ \cite{NSScS99}, Section~12.2.1.
\end{proof}

\begin{lemma}\label{lem2}
The mapping
$$
   KI:\H^s(M\times \mathbb{R}) \lra  H^s(M_\mathbb{Z}\times \mathbb{S}^1,\gamma) 
$$
is an isomorphism for all $s$.
\end{lemma}
\begin{proof}
We first note that $KI$ takes the rapidly decreasing functions on $M\times
\mathbb R$ to the smooth sections of $\gamma$.  The assertion then 
follows from the fact that this mapping takes the operator  
$$
\Delta_x+t^2-\frac{\partial^2 }{\partial t^2},
$$
which is the base for  the norm in   $\H^s(M\times \mathbb{R}) $, 
to the Laplacian (modulo lower order terms)
$$
\Delta_x-\frac{\partial^2 }{\partial \varphi^2}  -\frac{\partial^2 }{\partial t^2},
$$
which induces the norm in  $H^s(M_\mathbb{Z}\times \mathbb{S}^1,\gamma)$. 
\end{proof}

So,    $I=K^{-1}(KI)$ is an isomorphism as the composition of 
two isomorphisms. 

\paragraph{2. Computation of the operator-valued symbol.}

Let
$$
B=\sum_k B_k\left(x,-i\frac{\partial}{\partial x},t,-i\frac{\partial}{\partial t}\right)\widetilde{T}^k
$$
be a differential operator with shifts on $M\times \mathbb R$.
We have
\begin{multline*}
 (u(x,t,n))_n 
 \stackrel{I^{-1}}\longmapsto u(g^{-[t]}(x),\{t\},[t]) \stackrel{B}\longmapsto\\
\sum_k B_k\left(x,-i\frac{\partial}{\partial x},t,-i\frac{\partial}{\partial t}\right)
u(g^{-[t]}(x),\{t\},[t+k])
\stackrel{I}\longmapsto\\
\sum_k B_k\left(g^n(x),-i\frac{\partial}{\partial (g^n(x))},t+n,-i\frac{\partial}{\partial t}\right)\mathcal{T}^k u(x,t,n).
\end{multline*}
Thus, the operator $IBI^{-1}$ is indeed a differential operator on    
$M_\mathbb{Z}$ acting in sections of the infinite-dimensional bundle 
$\mathcal{E}$. 
Its symbol in the sense of \cite{NSScS99}, Section 3.3, is equal to
$$
\sigma(IBI^{-1})(x,\xi,t,\tau)=
\sum_k\sigma(B_k)(\partial g^n(x,\xi),t+n,\tau)\mathcal{T}^k:l^2(\mathbb{Z},\mu_{\xi,\tau,s})\to
l^2(\mathbb{Z},\mu_{\xi,\tau,s-1}).
$$
An easy computation shows that it smoothly depends on   $x,\xi,t,\tau$, 
satisfies the necessary estimates  
and has the compact variation property:  
The derivatives 
$$
\frac{\partial \sigma(IBI^{-1})}{\partial \xi},\;\; 
\frac{\partial \sigma(IBI^{-1})}{\partial \tau}:l^2(\mathbb{Z},\mu_{\xi,\tau,s})\to
l^2(\mathbb{Z},\mu_{\xi,\tau,s-1})
$$
are compact operators for all  $x,\xi,t,\tau$.

This gives us the desired mapping property 
\eqref{orb-symbol1} for the special case when  $B=\widetilde{D}\# A$.

\paragraph{3. Ellipticity condition.}

\begin{lemma}
Let $D$ be elliptic. Then the operator $\mathcal{D}$ is also elliptic, 
i.e., its symbol \eqref{orb-symbol1} is invertible provided that 
$|\xi|^2+\tau^2$ is large enough.
\end{lemma}
\begin{proof}
Consider the commutative diagram
$$
 \xymatrix{
    l^2(\mathbb{Z},\mu_{\xi,\tau,s}) \ar^{\sigma(\mathcal{D})(x,\xi,t,\tau)}[rr] \ar[d]_{\delta^{s}} & &
     l^2(\mathbb{Z},\mu_{\xi,\tau,s-1})  \ar[d]^{\delta^{s-1}}\\
   l^2(\mathbb{Z})  \ar[rr]_{\sigma_s(x,\xi,t,\tau)} &  & l^2(\mathbb{Z}),
 }
$$
where $\delta=(1+|\xi|^2+\tau^2+n^2)^{1/2}$.  In this diagram, the vertical mappings are isomorphisms, and the operator    $\sigma_s(x,\xi,t,\tau)$ is defined as
$$
 \sigma_s(x,\xi,t,\tau)=\delta^{s-1}\sigma(\mathcal{D})(x,\xi,t,\tau) \delta^{- s}.
$$ 
We claim that the operator
$$
\sigma_s(x,\xi,t,\tau):l^2(\mathbb{Z})\lra l^2(\mathbb{Z})
$$ 
is invertible for large $\xi,\tau$ and the norm of the inverse operator is uniformly bounded.  We have (the variables $x,\xi,t,\tau$ are omitted for brevity)
\begin{multline}\label{eq-est11}
\sigma_s=\delta^{s-1}\bigl(\sigma(D)\#\sigma(A)\bigr)\delta^{- s}=\bigl(\delta^{s-1}\sigma(D)\delta^{- s}\bigr)\#\delta^{-1}\sigma(A)=\\
=(\sigma(D)\delta^{-1} \# \delta^{-1}\sigma(A))+Q.
\end{multline}
Here and below $Q:l^2(\mathbb{Z})\to l^2(\mathbb{Z})$ stands for operator families of order $-1$ in the scale $l^2(\mathbb{Z},\mu_{\xi,\tau,s})$.  The first equality here is just the definition of the symbol  $\sigma_s$;
the second holds, because   $A$ has no shifts. 
The third follows from the fact that the commutator $[\mathcal{T},\delta^s]$ is an operator of order $\le s-1$:
\begin{multline*}
[\mathcal{T},\delta^s]
=(\mathcal{T}\delta^s\mathcal{T}^{-1}-\delta^s) \mathcal {T}
=[(1+\xi^2+\tau^2+(n+1)^2)^{s/2}-(1+\xi^2+\tau^2+n^2)^{s/2}]\mathcal{T}\\
=\delta^s\left(\left(1+\frac{2n+1}{1+\xi^2+\tau^2+n^2} \right)^{s/2}-1\right)
\mathcal{T}=
\delta^{s-1 }P,
\end{multline*}
where $P$ is a uniformly bounded operator family. 
Here we have used a Taylor expansion of   $(1+\alpha)^{s/2}$ at $\alpha=0$ for the last equality.
With a slight modification of the argument, we could also have obtained the form 
$\delta^{-1}\sigma(D) \# \sigma(A)\delta^{-1}+Q$ in  \eqref{eq-est11}.

To construct a left inverse of  $\sigma_s$, consider the operator $\sigma_s^*\sigma_s$.  A direct computation using \eqref{eq-est11} shows that
\begin{equation*}
\sigma_s^*\sigma_s=
\left(
  \begin{array}{cc}
     \delta^{-1}(\sigma(D)^*\sigma(D)+(n^2+\tau^2))\delta^{-1}\!\!\!\!\!\!\! & 0 \\
    0 & \!\!\!\!\!\!\! \delta^{-1}(\sigma(D)\sigma(D)^*+(n^2+\tau^2))\delta^{-1}
  \end{array}
\right)+Q.
\end{equation*}
Let us prove that the self-adjoint operators on the diagonal of this matrix are positive-definite. Indeed, consider for instance the operator in the left upper corner. Using the ellipticity of $D$ we have  with suitably small $c,c'>0$
\begin{multline*}
 \bigl( \delta^{-1 }(\sigma(D)^*\sigma(D)+n^2+\tau^2)\delta^{-1}w,w \bigr)
= \bigl((\sigma(D)^*\sigma(D)+n^2+\tau^2)\delta^{-1}w, \delta^{-1}w\bigr) \\
\ge c|\xi|^2\bigl(\delta^{-1}w, \delta^{-1}w\bigr)
+\bigl((n^2+\tau^2)\delta^{-1}w, \delta^{-1}w\bigr)
\ge c'\bigl(\delta^2 \delta^{-1}w, \delta^{-1}w\bigr)=c'(w,w).
\end{multline*}
So, $ \delta^{-1}(\sigma(D)^*\sigma(D)+n^2+\tau^2)\delta^{-1}$ 
is positive definite, and the operator   $\sigma_s^*\sigma_s$ is the sum of an 
invertible family with uniformly bounded inverse and an operator family
which tends to zero as  $(\xi,\tau)\to\infty$. 
Hence, the family
  $\sigma_s^*\sigma_s$ is uniformly invertible and   
  $(\sigma^*_s\sigma_s)^{-1}\sigma_s^*$ is a left inverse of  $\sigma_s$.

To construct a right inverse for $\sigma_s$, we consider the family $\sigma_s\sigma_s^*$. A similar reasoning shows that it is also uniformly invertible. Therefore, a right inverse of   $\sigma_s$  is equal to $\sigma_s^*(\sigma_s\sigma_s^*)^{-1}$.

The proof of the lemma is complete.
\end{proof}

\section{Equality of the Indices}\label{secBB}

Let us prove equality \eqref{eq-eq1}. 
Since $I$ is an isomorphism, the equality of the indices is a 
corollary of the following proposition.
\begin{proposition}
The operator  $\widetilde{D}\# A$ in \eqref{eq-product1} is Fredholm for all $s$, its index does not depend on $s$, and we have the equality
\begin{equation}\label{eq-eq2}
 \ind (\widetilde{D}\# A)=\ind D.
\end{equation}
\end{proposition}
\begin{proof}
Step 1. The operator $\widetilde{D}\# A$ is an operator with shifts on the cylinder; its 
symbol is 
$$
\sigma(\widetilde{D}\# A)(x,\xi,t,\tau)=\sigma(D)(x,\xi)\#\sigma(A)(t,\tau):
l^2(\mathbb{Z},\mathbb{C}^2)\lra  l^2(\mathbb{Z},\mathbb{C}^2),
$$
cf.\ the Appendix. 
Since $D$ and $A$ are elliptic,  $\sigma(D)(x,\xi)$ is invertible 
whenever $\xi\ne 0$, and  $\sigma(A)(t,\tau)$ is invertible 
whenever $t^2+\tau^2\ne 0$. 
Hence, the external product of these symbols is invertible whenever 
$|\xi|^2+t^2+\tau^2\ne 0$.  
Therefore, $\widetilde{D}\# A$ is elliptic and thus a Fredholm operator.

Step 2. Consider the operator family
$$
B_\varepsilon=\Big(\sum _k D_k \widetilde{T}^k_\varepsilon\Big)\# A:
\H^s(M\times \mathbb R, \mathbb C^2) 
\lra \H^{s-1}(M\times \mathbb R,\mathbb C^2),
$$
where $\widetilde{T}_\varepsilon u(x,t)=u(g(x),t+\varepsilon)$.  By construction
$$
 B_1=\widetilde{D}\# A.
$$
For a parametrix construction, we can associate to $B_\varepsilon$ 
the same symbol \eqref{shiftsymb}, independent of $\varepsilon$. 
This makes sense as the results in the Appendix are independent of the 
size of the shift. 
Then it is clear that the operators $B_\varepsilon$ are elliptic 
and hence Fredholm operators for  $\varepsilon\in [0,1]$.
As $D$ and $A$ commute, the technique of external products, 
cf. Section 9 in  \cite{AtSi1}, shows that 
$$
 \ker B_0=\ker B_0^*B_0=\ker D\otimes \ker A\simeq \ker D
$$ 
and similarly
$$
 \coker B_0=\ker B_0B_0^*=\ker D^*\otimes \ker A\simeq \coker D.
$$ 

Step 3. 
Thus, to prove   equality  \eqref{eq-eq1}, it remains to show that the index of    $B_\varepsilon$ does not depend on $\varepsilon$. 

This is not trivial, as the mapping   
$\varepsilon \mapsto B_\varepsilon$ is not continuous in operator norm. 
For the proof, we rely on Theorem 19.1.10 in   H\"ormander \cite{Hor3}.
We will proceed as follows:
The family $B_\varepsilon$ consists of Fredholm operators and is 
strongly continuous in  $\varepsilon$. 
Using a result of Schweitzer \cite{Schwe1}, we will show that 
the inverse symbol $\sigma(B_\varepsilon)^{-1}$
is represented by a series which converges rapidly, uniformly in $\varepsilon$.  
From this we infer the strong continuity of the associated operator family 
$B_\varepsilon^{-1}$.
Moreover, the families 
\begin{equation}\label{eq-ugly1}
K_{1,\varepsilon}=1-B_\varepsilon B_\varepsilon^{-1}
\quad\text{and}\quad 
K_{2,\varepsilon} =1-B_\varepsilon^{-1}B_\varepsilon
\end{equation}
consist of compact operators and are norm continuous.
This implies the uniform compactness condition in  H\"ormander's theorem. 

Here are the details: 
We represent the symbol of the operator with shifts as an element of the 
subalgebra of rapidly decreasing sequences with values in the algebra 
$\mathcal A$ of all smooth functions on the sphere $\{|\xi|^2+t^2+\tau^2=1\}$
in the crossed product algebra  $ \mathcal B\rtimes\mathbb Z$, 
where 
$\mathcal B $ denotes the continuous functions on that sphere and $\mathbb Z$ 
acts (abstractly) by powers of $T$. 
We then write
$$
\sigma(B_\varepsilon)=\sum_k b_k T^k;
$$
here the sum is finite and the right-hand side does not depend on  
$\varepsilon$. 
By Theorem 6.7 in \cite{Schwe1} (for the required strong spectral invariance
see Remark  I.1.14 in \cite{NaSaSt17}), the inverse symbol is equal to
$$
\sigma(B_\varepsilon)^{-1}=\sum_k a_k T^k, 
$$
where the coefficients  $a_k$ are smooth symbols, 
whose seminorms all tend to zero faster than any power of $k$.
Let
$$
B_\varepsilon^{-1}=\sum_k \widehat{a_k} \widetilde{T}^k_\varepsilon,
$$
where the operators $\widehat{a_k}$ are chosen such that all their seminorms 
are rapidly decreasing in $k$. 
A direct computation shows the estimate
$\|\widetilde{T}^k_\varepsilon\|\le C(1+|k|)^{|s|}$ for the norm of 
$\widetilde T^k_\varepsilon$ on $\H^s(M\times \mathbb R)$.
Hence the series for $B_\varepsilon^{-1}$ converges.
Moreover, the strong continuity of $T_\varepsilon$ and the rapid decay of the
$\widehat {a_k}$ imply  the strong continuity of  
$\varepsilon\mapsto B^{-1}_\varepsilon$. 

It remains to consider the families  \eqref{eq-ugly1}. 
Let us first consider $K_{1,\varepsilon}$.  
We infer from Cauchy's product formula that
$$
1-B_\varepsilon B_\varepsilon^{-1}=
1-  \sum_{n+k=0}  \widehat{b_n} 
     \widetilde{T}^n_\varepsilon \widehat{a_k} \widetilde{T}^{-n}_\varepsilon 
 -
    \sum_{l\not=0}
     \left(  \sum_{n+k=l}  \widehat{b_n} 
     \widetilde{T}^n_\varepsilon \widehat{a_k} \widetilde{T}^{-n}_\varepsilon \right) 
     \widetilde{T}^{l}_\varepsilon .
$$
By construction, the coefficient of each power of $\widetilde T_\varepsilon$ 
is a compact operator. 
As a consequence, the sum of the 
first two terms on the right hand side is compact, and so is each of the terms
$L_{l,\varepsilon}= \sum_{n+k=l}  \widehat{b_n} 
     \widetilde{T}^n_\varepsilon \widehat{a_k} \widetilde{T}^{-n}_\varepsilon$,
$l\not=0$.     
Now we observe that 
$$\varepsilon \mapsto  \widetilde{T}^n_\varepsilon \widehat{a_k} \widetilde{T}^{-n}_\varepsilon$$
is continuous. Since we only have finitely many summands  for each $l$, 
the sum of the first two terms is a continuous function of $\varepsilon$, 
and so is each $L_{l,\varepsilon}$. 
Using the strong continuity of $\widetilde T_\varepsilon$ we conclude
that all the mappings 
$\varepsilon\mapsto L_{1,\varepsilon}\widetilde T_\varepsilon^l$ 
are continuous in norm and compact. 
The rapid decay of the sequence $(\widehat{b_k})$ then
shows that $\varepsilon\mapsto K_{1,\varepsilon}$ is compact  and norm 
continuous. 
Hence the conditions of H\"ormander's theorem are fulfilled and   
$
\ind B_\varepsilon$ is constant. This completes the proof.
\end{proof}

\section{The Index of Operators in Hilbert bundles}\label{sec2}

In this section, we prove Theorem  \ref{th-luke1a} on the index of operator $\mathcal{D}$
with operator-valued symbol.

\paragraph{1. Reduction to a pseudodifferential operator  with homogeneous symbol.}
Let $\mathcal{D}'$ be a pseudodifferential operator of order zero on  
$M_\mathbb{Z}$, whose symbol is equal to $\sigma(\mathcal{D})$ for 
$|(\xi,\tau)|=R$, 
where $R$ is a sufficiently large positive number, and is a homogeneous function of degree zero in the covariables  $(\xi,\tau)$. 
Explicitly this symbol is defined by the formula  
$$
\sigma(\mathcal{D}')(x,\xi,t,\tau)=\sigma(\mathcal{D})
\Big(x,\frac{R\xi}{|(\xi,\tau)|},t,\frac{R\tau}{|(\xi,\tau)|}\Big).
$$ 
\begin{lemma}
One has the equality
\begin{equation}\label{eq-indr1}
\ind \mathcal{D}=\ind \mathcal{D}'.
\end{equation}
\end{lemma}
\begin{proof}
Let $\Lambda$ be a first-order operator with the symbol   
$(|\xi|^2+\tau^2+n^2)^1/2$.
The index of this operator is obviously equal to zero. We obtain
$$
\ind  \mathcal{D}=\ind  (\mathcal{D}\Lambda^{-1})=
p_![\sigma(\mathcal{D}\Lambda^{-1})]=p_![\sigma(\mathcal{D})]=p_![\sigma(\mathcal{D}')]=\ind \mathcal{D}' ,
$$
where $p:T^*M_{\mathbb{Z}^N}\to pt$ is the projection, $[\sigma(\mathcal{D})]\in K^0(T^*M_{\mathbb{Z}})$
is the class of symbol in $K$-theory. Here the first equality follows from the fact that   $\Lambda$ has index zero, the second and last equalities follow from the index formula in  $K$-theory (see \cite{Luk1,NSScS15}). 
\end{proof}
\paragraph{2. The topological index of operator-valued symbols.}

We prove the index formula for pseudodifferential operators with operator-valued symbols in the following class.

\begin{definition}\label{def-cv2}
An element $\sigma\in C^\infty(S^*M_{\mathbb{Z}},\End \mathcal{E})$ is a  {\em symbol} if its derivatives  
\begin{equation}\label{cv-2}
\frac{\partial\sigma}{\partial \xi},\;\frac{\partial\sigma}{\partial t},\;
\frac{\partial\sigma}{\partial \tau}:l^2(\mathbb{Z})\lra l^2(\mathbb{Z})
\end{equation}
are operators of order $-1$ in the scale of spaces $l^2(\mathbb{Z},\mu_s)$ defined using the measures 
$$
\mu_s(n)=(1+n^2)^s.
$$ 
A symbol is {\em elliptic} if it is invertible at each point on  $S^*M_{\mathbb{Z}}$.
\end{definition}
The symbol of the operator $\mathcal{D}'$ is an example for a symbol of this kind.
Note also that due to the assumptions on the derivatives, 
these are symbols in the sense of Luke  \cite{Luk1}, 
i.e., they have compact variation in the covariables.   

We shall now analyze the topological index of Definition \ref{def-4z}
for symbols in this class. The following lemma establishes the essential
properties of the integrand. 
\begin{lemma}\label{prop-tind}
Let $\omega\in\Lambda^{2k}(M)$  be a closed $\mathbb{Z}$-invariant form of degree $2k$. Then
\begin{enumerate}
\renewcommand{\labelenumi}{{\rm (\alph{enumi})}}
\item  The expression
\begin{equation}\label{eq-form1}
 \tr_\mathcal{E}
\left[(\sigma ^{-1}d\sigma)^{2j-1}\omega \right] \in \Lambda^{2j-1+2k}(S^*  M_{\mathbb{Z}})
\end{equation}
is a well-defined smooth closed form  if  $2k+2j-1> n+1$;
\item The cohomology class of \eqref{eq-form1}  is invariant with respect to homotopies of the elliptic symbol   $\sigma$ if $2k+2j-1> n+2$. 
\end{enumerate}
\end{lemma}
\begin{proof}
(a) Since $\omega$ is $\mathbb{Z} $-invariant, it follows that we can consider this form as a smooth form  on $ S^*  M_{\mathbb{Z} }$. It suffices to prove the statement in local coordinates, which we denote by  $x,\xi,t,\tau$. 
Let us decompose the form $(\sigma ^{-1}d\sigma)^{2j-1}$ into the sum of monomials
\begin{eqnarray}
\label{mono}
(\sigma ^{-1}d\sigma)^{2j-1}=\sum_{\alpha,\beta,\gamma,\delta} f_{\alpha\beta\gamma\delta}(x,\xi,t,\tau)
dx^\alpha dt^\beta d\xi^\gamma d\tau^\delta,
\end{eqnarray}
where the summation is over the multi-indices  $\alpha,\beta,\gamma,\delta.$ Let $\mu=|\beta|+|\gamma|+|\delta|$. 
Then 
$$
 |\alpha|+\mu=2j-1,\quad |\alpha|+2k\le n.
$$
This inequality means simply that a nonzero monomial can have at most $n$ differentials $dx$.
It follows that $\mu\ge 2 k-n+2j-1$. By \eqref{cv-2} this implies  that the product   
$(\sigma ^{-1}d\sigma)^{2j-1}\omega$ is an operator of order $-\mu$.  Therefore, if 
$$
2k-n+2j-1>1,
$$
then this product is a form, which ranges in trace class operators and the form  
\eqref{eq-form1} is well defined. The proof that the form is closed is  standard,   see  e.g.\ \cite{Zha3}: 
\begin{multline*}
d\Bigl(\tr_\mathcal{E}((\sigma^{-1}d\sigma)^{2j-1}\omega)\Bigr)=(2j-1)\tr_\mathcal{E}\Bigl(d(\sigma^{-1}d\sigma)(\sigma^{-1}d\sigma)^{2j-2}\omega\Bigr)
=\\
=-(2j-1)\tr_\mathcal{E}((\sigma^{-1}d\sigma)^{2j}\omega)=(2j-1)\tr_\mathcal{E}((\sigma^{-1}d\sigma)^{2j}\omega)=0.
\end{multline*}

(b) Let $\sigma=\sigma_\varepsilon$  be a homotopy  with parameter $\varepsilon\in[0,1]$ of invertible elements.
Then the  corresponding forms    $\tr_\mathcal{E}((\sigma^{-1}d\sigma)^{2j-1}\omega)$ define the same cohomology class. This follows from the transgression formula  
$$
\frac\partial{\partial \varepsilon}
\tr_\mathcal{E}\Bigl((\sigma^{-1}d\sigma)^{2j-1}\omega\Bigr)=(2j-1)d\tr_\mathcal{E}\Bigl(\sigma^{-1}\frac{\partial
\sigma}{\partial \varepsilon}(\sigma^{-1}d\sigma)^{2j-2}\omega\Bigr),
$$
which is valid if  $2k+2j-1>n+2$. 
\end{proof}

The integral of expression \eqref{eq-form1} over  $S^*M_{\mathbb Z}$  
can only be non-zero, if in \eqref{mono} 
$\beta+|\gamma|+\delta\ge n+1$, and hence $2j-1\ge n+1$.
Moreover, if $n>1$, then $n+2<2n+1$, so that we necessarily have $k>0$ 
in case $2j-1\le n+2$. 
We thus conclude from Lemma \ref{prop-tind}:
 
\begin{corollary}
The topological index 
is a well-defined homotopy invariant provided that  $n>1.$
\end{corollary}

\paragraph{2. Index theorem.}  We next establish a cohomological index formula in the spirit of Rozenblum  \cite{Roz4,Roz3} for our situation. 
\begin{theorem}\label{th-luke1}
The Fredholm index of an elliptic pseudodifferential operator  
$$
 \widehat{\sigma}:L^2(M_{\mathbb{Z} },\mathcal{E})\to  L^2(M_{\mathbb{Z} },\mathcal{E})
$$ 
with symbol $\sigma$ is equal to
\begin{equation}\label{ind-luke1}
 \ind \widehat{\sigma}=\ind_t \widehat{\sigma},
\end{equation}
where the topological index  $\ind_t \widehat{\sigma}$ is defined in \eqref{eq-indt2c}.
\end{theorem}
\begin{proof}
Step 1. Let us fix a covector $\tau_0\ne 0$ and let $\sigma'$ be the elliptic symbol
$$
\sigma'(x,\xi,t,\tau)=\sigma(x,\xi,t,\tau)\sigma^{-1}(x,0,t,\tau_0).
$$
One has
\begin{equation}\label{eq-triv1}
 \ind \widehat{\sigma}= \ind \widehat{\sigma'},\quad 
 \ind_t \widehat{\sigma}= \ind_t  \widehat{\sigma'}.
\end{equation}
Both equalities follow from the logarithmic property of the (topological) index 
and the fact that  
$\sigma(x,0,t,\tau_0)$ is the symbol of a multiplication operator.

Step 2. To prove the theorem, it suffices by \eqref{eq-triv1}  to check that 
\eqref{ind-luke1} holds for the symbol $\sigma'$. 
This symbol is equal to the sum of the identity symbol and the symbol  $\sigma'-1$.
The last symbol can be arbitrarily well approximated by a symbol $\sigma''$ ranging in finite rank operators. Then we get
$$
\ind \widehat{\sigma'}=\ind \widehat{1+\sigma''}
=\ind_t \widehat{1+\sigma''}=\ind_t \widehat{\sigma'}.
$$
Here the first and third equalities are valid, because   $\sigma'$ and  $1+\sigma''$ are homotopic.
The second equality follows from the Atiyah--Singer formula, since the operator  $1+\sigma''$ 
is a direct sum of the identity operator and an elliptic pseudodifferential operator 
acting in sections of finite-dimensional bundles.  This completes the proof.
\end{proof}

Now the index theorem~\ref{th-luke1a} follows from Equation  \eqref{eq-indr1} and Theorem~\ref{th-luke1}.

\section*{Appendix. Operators on the Infinite Cylinder}
\addtocounter{section}{1}
\setcounter{equation}{0}

\addcontentsline{toc}{section}{Appendix. Operators on the Infinite Cylinder}

\paragraph{Differential operators.}

For a smooth closed manifold $M$ consider the cylinder   $M\times\mathbb{R}$ with coordinates $x,t$; denote the dual coordinates  by   $\xi,\tau$. 

We define the Sobolev space $\H^s(M\times \mathbb{R})$  
as the completion of the Schwartz space  $\mathcal {S}(M\times \mathbb{R})$ 
on the cylinder $M\times \mathbb{R}$ with respect to the norm
$$
 \|u\| _s=\bigl\|(1+\Delta_x+t^2-\partial^2_t)^{s/2}u\bigr\|_{L^2(M\times\mathbb{R})}.
$$
We consider differential operators of the form  
\begin{equation}\label{eq-pdo1}
 A=a\left(x,-i\frac\partial{\partial x},t,-i\frac\partial{\partial t}\right):
\H^s(M\times \mathbb{R})\lra \H^{s-m}(M\times \mathbb{R}),
\end{equation}
where the symbol $a(x,\xi,t,\tau)$ is a smooth function on $T^* (M\times \mathbb{R})$,  which is a homogeneous polynomial
of degree $m$  in the variables $(\xi,t,\tau)$: 
$$
a(x,\lambda \xi,\lambda t,\lambda \tau)=\lambda^m a(x,\xi,t,\tau),\qquad \forall \lambda>0.
$$
The operator  \eqref{eq-pdo1} is determined by the symbol in a standard way:  
$M$  is covered by coordinate neighborhoods  $U_j$. 
This gives a covering of the cylinder by the domains $U_j\times \mathbb{R}$. 
Then in each neighborhood the symbol is quantized in the standard way. 
Finally, using a partition of unity, we patch the local operators to a global operator.

\begin{proposition}
 The operator \eqref{eq-pdo1} is bounded. 
The mapping $a\mapsto A$  is a homomorphism of algebras up to operators of lower order.
If $A$ is elliptic (i.e., $a$ is invertible for $\xi^2+\tau^2+t^2\ne 0$), then $A$ is Fredholm.
\end{proposition}

\begin{proof}
Let us consider the cylinder as the total space of the covering  
$$
M\times \mathbb{R}\lra M\times \mathbb{S}^1,
$$
with base $M\times \mathbb{S}^1$.  
The Schwartz space  $\mathcal S(M\times \mathbb{R})$  is isomorphic to the space of smooth sections of a (nontrivial) vector bundle on the base $M\times \mathbb{S}^1$, whose fiber is the Schwartz space   
$\mathcal S(\mathbb{Z})$
of rapidly decaying sequences    (i.e., functions on the fiber of the covering).  
Let us now take the Fourier transform 
$$
\mathcal S(\mathbb{Z})\stackrel{\mathcal{F}_{n\to\varphi}}
\lra C^\infty(\mathbb{S}^1)
$$
in each fiber. As a result, we obtain the space of smooth sections of a one-dimensional complex vector bundle on   $M \times \mathbb{T}^2$, on which
we choose coordinates  $x,t,\varphi$ such that $t\in[0,1]$, $\varphi\in[0,2\pi]$ .
Collecting the transformations,  
we obtain   (cf.~\cite{NaSaSt17}):

\begin{lemma} 
We have an isomorphism
\begin{equation}\label{eq-isom1}
\mathcal  S(M\times \mathbb{R})\simeq C^\infty(M \times \mathbb{T}^2,\gamma)
\end{equation}
between the Schwartz space on the cylinder and the space  
\begin{equation*}\label{bottbundle}
C^\infty(M \times \mathbb{T}^2,\gamma)=\{ v\in C^\infty(M\times \mathbb{R}\times \mathbb{S}^1)
\; |\; v(x,t+1,\varphi)=v(x,t,\varphi)e^{-  i\varphi}\}.
\end{equation*}
of smooth sections of the one-dimensional vector bundle    $\gamma$ on $M \times \mathbb{T}^2$. 
It is given by 
$$
f(x,t)\longmapsto \frac 1{\sqrt{2\pi}}\sum_{n\in\mathbb{Z}} f(x,t+n)e^{ i n\varphi}.
$$
\end{lemma}

\begin{proof}
The inverse mapping is equal to
$$
v(x,t,\varphi)\longmapsto \frac 1
{\sqrt{2\pi}}\int_{\mathbb{S}^1}v(x, \{t\},\varphi)e^{-i[t]\varphi}d\varphi,
$$
where $[a]$ and $\{a\}$ stand for the integer and fractional parts of a real number   $a$.
\end{proof}
The isomorphism \eqref{eq-isom1} takes the operators
$$
x,\quad \frac{\partial}{\partial x},\quad t,\quad \frac{\partial}{\partial t}
$$
to the operators  
$$
x,\quad \frac{\partial}{\partial x},\quad  -i\frac{\partial}{\partial \varphi}+t,\quad \frac{\partial}{\partial t}.
$$
Hence, the operator on the cylinder 
$$
\Delta_x+t^2-\frac{\partial^2}{\partial t^2},
$$
which defines the scale of Sobolev spaces, is transformed   
(modulo lower order terms) into the Laplace operator
$$
\Delta_x-\frac{\partial^2}{\partial \varphi^2}-\frac{\partial^2}{\partial t^2}
$$
on the sections of $\gamma$. 
It follows that the mapping~\eqref{eq-isom1} defines an isomorphism 
$$
J:\H^s(M\times\mathbb{R})\lra H^s(M\times \mathbb{T}^2,\gamma)
$$
between the 
Sobolev spaces on the cylinder $M\times\mathbb{R}$ and on the closed manifold $M\times \mathbb{T}^2$, respectively.
Moreover, this mapping takes differential operators on the cylinder to differential operators on $M\times \mathbb{T}^2$. 
It is easy to check that the symbols of these differential operators are related by the equality  
$$
 \sigma(JAJ^{-1})(x,\xi,t,\tau,\varphi,\eta)=
  \sigma(A)(x,\xi,\eta,\tau).
$$
Now the desired statement follows from standard results for (pseudo)differential operators on the closed manifold  $M\times \mathbb{T}^2$.
\end{proof}

\paragraph{Operators with shifts.}
  
Consider the operator with shifts on $M\times \mathbb R$:
\begin{eqnarray}\label{opshift}
\widetilde D=\sum_k D_k\left(x,-i\frac\partial{\partial x},t,-i\frac\partial{\partial t}\right)\widetilde{T}^k:
\H^s(M\times \mathbb{R})\lra \H^{s-m}(M\times \mathbb{R}),
\end{eqnarray}
where the $D_k$ are differential operators of orders $\le m$ on the cylinder
in the above sense. 
\begin{definition}
The {\em symbol} of the operator $\widetilde D$ is the operator
\begin{eqnarray}\label{shiftsymb}
\sigma(\widetilde D)(x,\xi,t,\tau)=\sum_k  \sigma(D_k)(\partial g^n(x,\xi),t,\tau)\mathcal{T}^k:l^2(\mathbb{Z})\lra l^2(\mathbb{Z}),
\end{eqnarray}
where $\sigma(D_k)$ is the principal symbol of order $m$ of $D_k$. 
                     
We call $\tilde D$ {\em elliptic}, if $\sigma(\widetilde D)$ 
is invertible for $(\xi,t,\tau)\not=0$.
\end{definition}

This definition is motivated by the fact that   $\widetilde{T}$ is almost unitary in 
$\H^s(M\times \mathbb{R})$, i.e., it is unitary modulo compact operators. 
\begin{proposition}
If $\widetilde D$ is elliptic, then it is a Fredholm operator. 
Its index, kernel and cokernel do not depend on  $s$.
\end{proposition}

\begin{proof}
The Fredholm property is established in a standard way  (see e.g. \cite{AnLe1}).  
The independence of the index of $s$ is proven using the fact that 
$D$ almost commutes with order reduction operators. 
The independence of the kernel and the cokernel of $s$ follows 
from the fact that, on the one hand, the dimension of the kernel is a 
nonincreasing function of $s$, while the dimension of the cokernel is 
nondecreasing, and, on the other hand, the difference of these dimensions is 
equal to the index, which is constant in   $s$. 
\end{proof}



\begin{thebibliography}{NSSS05b}

\bibitem[AL94]{AnLe1}
A.~Antonevich and A.~Lebedev.
\newblock {\em Functional-Differential Equations. {I. $C\sp *$}-theory}.
\newblock Number~70 in Pitman Monographs and Surveys in Pure and Applied
  Mathematics. Longman, Harlow, 1994.

\bibitem[AS68]{AtSi1}
M.~F. Atiyah and I.~M. Singer.
\newblock The index of elliptic operators {I}.
\newblock {\em Ann. of Math.}, 87:484--530, 1968.

\bibitem[BC88]{BaCo2}
P.~Baum and A.~Connes.
\newblock Chern character for discrete groups.
\newblock In {\em A f\^ete of topology}, pages 163--232. Academic Press,
  Boston, MA, 1988.

\bibitem[HK01]{HiKo1}
N.~Higson and G.~Kasparov.
\newblock {$E$}-theory and {$KK$}-theory for groups which act properly and
  isometrically on {H}ilbert space.
\newblock {\em Invent. Math.}, 144(1):23--74, 2001.

\bibitem[H{\"o}r85]{Hor3}
L.~H{\"o}rmander.
\newblock {\em The Analysis of Linear Partial Differential Operators. {III}}.
\newblock Springer--Verlag, Berlin Heidelberg New York Tokyo, 1985.

\bibitem[Luk72]{Luk1}
G.~Luke.
\newblock Pseudodifferential operators on {H}ilbert bundles.
\newblock {\em J. Diff. Equations}, 12:566--589, 1972.

\bibitem[NSS08]{NaSaSt17}
V.~E. Nazaikinskii, A.~Yu. Savin, and B.~Yu. Sternin.
\newblock {\em Elliptic Theory and Noncommutative Geometry}.
\newblock Birkh\"auser Verlag, Basel, 2008.

\bibitem[NSSS05a]{NSScS99}
V.~Nazaikinskii, A.~Savin, B.-W. Schulze, and B.~Sternin.
\newblock {\em Elliptic Theory on Singular Manifolds}.
\newblock CRC-Press, Boca Raton, 2005.

\bibitem[NSSS05b]{NSScS15}
V.~Nazaikinskii, A.~Savin, B.-W. Schulze, and B.~Sternin.
\newblock On the index of differential operators on manifolds with edges.
\newblock {\em Sbornik: Mathematics}, 196(9):1271--1305, 2005.

\bibitem[Roz02]{Roz4}
G.~Rozenblum.
\newblock On some analytical index formulas related to operator-valued symbols.
\newblock {\em Electron. J. Differential Equations}, (17):1--31, 2002.

\bibitem[Roz03]{Roz3}
G.~Rozenblum.
\newblock Regularisation of secondary characteristic classes and unusual index
  formulas for operator-valued symbols.
\newblock In {\em Nonlinear hyperbolic equations, spectral theory, and wavelet
  transformations}, volume 145 of {\em Oper. Theory Adv. Appl.}, pages
  419--437. Birkh\"auser, Basel, 2003.

\bibitem[Sch93]{Schwe1}
L.~B. Schweitzer.
\newblock Spectral invariance of dense subalgebras of operator algebras.
\newblock {\em Internat. J. Math.}, 4(2):289--317, 1993.

\bibitem[Zha01]{Zha3}
W.~Zhang.
\newblock {\em Lectures on {C}hern-{W}eil theory and {W}itten deformations}.
\newblock World Scientific Publishing Co. Inc., River Edge, NJ, 2001.

\end{thebibliography}
\end{document}